\documentclass[11pt,english]{article}

\pdfoutput=1

\usepackage[margin = 1.1in]{geometry}

\usepackage{amsthm}
\usepackage{amsmath}
\usepackage{amssymb}
\usepackage{setspace}
\usepackage{mathtools}
\usepackage{graphicx}
\usepackage[hidelinks]{hyperref}
\usepackage{thm-restate}
\usepackage{cleveref}
\usepackage{comment}
\usepackage{enumerate}
\usepackage{framed}
\usepackage{subcaption}

\theoremstyle{plain}

\newtheorem*{thm*}{Theorem}
\newtheorem{thm}{Theorem}[section]
\Crefname{thm}{Theorem}{Theorems}

\newtheorem*{lem*}{Lemma}
\newtheorem{lem}[thm]{Lemma}
\Crefname{lem}{Lemma}{Lemmas}

\newtheorem*{claim*}{Claim}
\newtheorem{claim}[thm]{Claim}
\crefname{claim}{Claim}{Claims}
\Crefname{claim}{Claim}{Claims}

\Crefname{prop}{Proposition}{Propositions}

\newtheorem{cor}[thm]{Corollary}
\Crefname{cor}{Corollary}{Corollaries}

\newtheorem{conj}[thm]{Conjecture}
\Crefname{conj}{Conjecture}{Conjectures}

\newtheorem{qn}[thm]{Question}
\Crefname{qn}{Question}{Questions}

\Crefname{obs}{Observation}{Observations}

\theoremstyle{definition}

\Crefname{prob}{Problem}{Problems}

\Crefname{defn}{Definition}{Definitions}

\newtheorem{example}[thm]{Example}
\Crefname{example}{Example}{Examples}

\theoremstyle{remark}

\renewenvironment{proof}[1][]{\begin{trivlist}
\item[\hspace{\labelsep}{\bf\noindent Proof#1.\/}] }{\qed\end{trivlist}}

\newcommand{\mF}{\mathcal{F}}
\newcommand{\mA}{\mathcal{A}}
\newcommand{\mB}{\mathcal{B}}
\newcommand{\mP}{\mathcal{P}}

\newcommand{\mT}{\mathcal{T}}
\newcommand{\mX}{\mathcal{X}}
\newcommand{\mY}{\mathcal{Y}}
\newcommand{\p}{\partial}

\newcommand{\remove}[1]{}

\begin{document}

\title{Disjoint pairs in set systems with restricted intersection}
\date{\vspace{-5ex}}

\author{
    Ant{\'o}nio Gir{\~a}o \thanks{
        School of Mathematics,
        University of Birmingham,
        Edgbaston, Birmingham, B15 2TT,
        United Kingdom;
        e-mail:
        \mbox{\texttt{giraoa@bham.ac.uk}}\,.
    } 
    \and
    Richard Snyder \thanks{
    	Karlsruhe Institute of Technology,
    	Karlsruhe, Germany;
    	e-mail: \mbox{\texttt{richard.snyder@kit.edu}}\,.
     }
}

\maketitle

\singlespace
\begin{abstract}

\setlength{\parskip}{\medskipamount}
    \setlength{\parindent}{0pt}
    \noindent
    
The problem of bounding the size of a set system under various intersection restrictions has a central place in extremal combinatorics. 
We investigate the maximum number of \emph{disjoint pairs} a set system can have in this setting. In particular, we show 
that for any pair of set systems $(\mA, \mB)$ which avoid a cross-intersection of size $t$, the number of disjoint pairs $(A, B)$ 
with $A \in \mA$ and $B \in \mB$ is at most $\sum_{k=0}^{t-1}\binom{n}{k}2^{n-k}$. This implies an asymptotically best possible upper 
bound on the number of disjoint pairs in a single $t$-avoiding family $\mF \subset \mP[n]$. We also study this problem 
when $\mA$, $\mB \subset [n]^{(r)}$ are both $r$-uniform, and show that it is closely related to the problem of determining the maximum of the product $|\mA||\mB|$ 
when $\mA$ and $\mB$ avoid a cross-intersection of size $t$, and $n \ge n_0(r, t)$.
        \end{abstract}

\onehalfspace
\section{Introduction}\label{sec:intro}
Let $[n]$ denote the set $\{1, \ldots , n\}$ and for two set systems $\mA$, $\mB \subset \mP[n]$, let
$d(\mA, \mB)$ denote the number of \emph{disjoint pairs}; that is, the number of pairs $(A, B) \in \mA \times \mB$ 
with $A \cap B = \varnothing$. Similarly, for a set system $\mF \subset \mP[n]$ we let $d(\mF)$ denote the number of disjoint 
pairs in $\mF$. Accordingly, $d(\mF) = \frac{1}{2}d(\mF, \mF)$ (unless, of course, $\varnothing \in \mF$, in which 
case $d(\mF) = \frac{1}{2}(d(\mF, \mF) - 1)$). We are interested in the maximum number of disjoint pairs a set system $\mF$ can have 
under certain restrictions on the possible intersection sizes of elements of $\mF$.
For a set $L$ of nonnegative integers, a set system $\mF$ is said to be $L$-\emph{intersecting} if $|F_1 \cap F_2| \in L$ for all distinct $F_1, F_2 \in \mF$. Similarly, a pair of set systems $(\mA, \mB)$ is $L$-\emph{cross-intersecting} if $|A\cap B| \in L$ whenever $A \in \mA, B \in \mB$. When $L = \{t, \ldots, n\}$ we say $\mF$ is $t$-\emph{intersecting}, and when $t = 1$ we shall simply say $\mF$ is \emph{intersecting}. Finally, if $L = [n]\setminus \{t\}$, we shall say that $\mF$ (resp., $(\mA, \mB)$) is $t$-\emph{avoiding} (resp., $t$-\emph{cross-avoiding}).

\subsection{Background}

The problem of bounding the size of a set system under certain intersection restrictions has a central place in Extremal Set Theory. We shall not give a full account of such problems, but only touch upon some results that are particularly relevant for our purposes (for a broader account we refer the interested reader to the recent survey of Frankl and Tokushige~\cite{FranklTokushige}). The Erd\H{o}s-Ko-Rado Theorem~\cite{EKR} is perhaps the most foundational result in this area, determining the maximum size of 
an intersecting $r$-uniform set system. More precisely, the theorem says that if $n \ge 2r$ and $\mF \subset [n]^{(r)}$ is an intersecting set system, 
then $|\mF| \le \binom{n-1}{r-1}$, and moreover, if $n > 2r$, then equality holds only when $\mF$ consists of all $r$-sets containing a fixed element of the ground set. Numerous extensions and variations have been addressed over the years. Perhaps most notably, The Complete Intersection Theorem of Ahlswede and Khachatrian~\cite{AK} determines the maximum size of a $t$-intersecting set system $\mF \subset [n]^{(r)}$ for all values of $n$. In the non-uniform case, Katona~\cite{Katona} showed that any $(t+1)$-intersecting set system $\mF$ satisfies 
\[
	|\mF| \leq |\mF(n, t)|,
\]
where $\mF(n, t)$ is $\{A: |A| \ge \frac{n + t +1}{2}\}$ if $n + t$ is odd, or $\{A:|A\cap([n]\setminus \{1\}))| \ge \frac{n+t}{2}\}$ if $n+t$ is even.
Trivially, if a set system is $(t+1)$-intersecting then it is also $t$-avoiding. Erd\H{o}s~\cite{erdos-proc} asked what happens 
when we weaken the condition that all $F_1, F_2 \in \mF$ satisfy $|F_1\cap F_2| > t$ to $|F_1 \cap F_2| \neq t$. Frankl and F{\"u}redi~\cite{FranklFuredinonuniform} answered this question, showing that 
when $n \ge n_0(t)$ we recover the same asymptotic solution as in Katona's theorem. In particular, letting 
\[
	\mF^*(n, t) = \mF(n, t) \cup [n]^{(\le t-1)},
\]
they showed that as long as $n \ge n_0(t)$ and $\mF \subset \mP([n])$ is $t$-avoiding, then $|\mF| \le |\mF^*(n, t)|$.

In this paper, instead of focusing on the \emph{size} of set systems with imposed intersection conditions, we are interested in the maximum number of disjoint pairs they can have. Alon and Frankl~\cite{FranklAlon} addressed the problem of determining the maximum number of disjoint pairs in a set system of fixed size. Obviously, we always have $d(\mF) < |\mF|^2$, but for large families they showed that this bound is far off: if $\mF$ has size $2^{n/2 + o(n)}$, then $d(\mF) = |\mF|^{2-o(1)}$. Problems concerning the minimum number of disjoint pairs in set systems have been studied by Ahlswede~\cite{Ahlswedemindisjoint}, Frankl~\cite{Frankldisjointpairsnonuniform}, Bollob{\'a}s and Leader~\cite{BollobasLeader}, and Das, Gan, and Sudakov~\cite{SudakovGanDas}.

Now, if $\mF \subset \mP[n]$ is $t$-intersecting for $t \ge 1$, then trivially $d(\mF) = 0$. However, what happens to the maximum number of disjoint pairs if we just impose that $\mF$ forbids a single (positive) intersection size? Our main line of enquiry 
investigates what happens to the parameter $d(\mF)$ under this weaker condition.


\subsection{Our Results}

Our first result provides an upper bound for the maximum number of disjoint pairs in $t$-avoiding set systems, for any 
$t \ge 1$.

\begin{restatable}{thm}{thmtavoidingsingle}\label{thm:t-avoiding-single}
Let $n, t$ be positive integers with $t \le n$ and suppose that $\mF \subset \mP[n]$ is $t$-avoiding. Then 
\[
	d(\mF) \leq \frac{1}{2}\left(\sum_{k=0}^{t-1} \binom{n}{k}2^{n-k} - 1\right).
\]
\end{restatable}

Note that the number of disjoint pairs in $\mF^*(n, t)$ is at least (assuming for simplicity that $n+ t$ is odd)
\begin{align*}
\sum_{k=0}^{t-1} \binom{n}{k}\cdot \sum_{j=(n+t+1)/2}^{n-k}\binom{n-k}{j}&=\sum_{k=0}^{t-1} \binom{n}{k} (1-o(1))2^{n-k-1}\\
&=(1-o(1))\frac{1}{2}\sum_{k=0}^{t-1} \binom{n}{k}2^{n-k},
\end{align*}
as $n \rightarrow \infty$.
Therefore, for large $n$ the upper bound we obtain in \Cref{thm:t-avoiding-single} is essentially best possible. We conjecture that $\mF^{*}(n,t)$ in fact maximizes the number of disjoint pairs for $t$-avoiding set systems (see \Cref{sec:final}).
We shall actually prove a `two-family' version which bounds the number of disjoint pairs in 
a pair $(\mA, \mB)$ of $t$-cross-avoiding set systems. In particular, 
\Cref{thm:t-avoiding-single} immediately follows from the following result.

\begin{restatable}{thm}{thmtavoidingdouble}\label{thm:t-avoiding-double}
Let $n, t$ be positive integers with $t \le n$ and suppose that $(\mA, \mB) \subset \mP[n] \times \mP[n]$ is a pair of $t$-cross-avoiding set systems. Then
\[
	d(\mA, \mB) \leq \sum_{k=0}^{t-1} \binom{n}{k}2^{n-k}.
\]

\end{restatable}

We remark that this is a generalization of a result in~\cite{C}, where the case $t= 1$ was established. We are also able to classify the extremal examples for \Cref{thm:t-avoiding-double}. Namely, if $t = 1$, then equality occurs if and only if $\mA = \mP(S), \mB = \mP([n]\setminus S)$ (or vice-versa) for some subset $S \subseteq [n]$, and if $t\geq 2$, equality holds if and only if $\mA = [n]^{(\le t-1)}, \mB = \mP[n]$ (or vice-versa).

\Cref{thm:t-avoiding-double} has the following immediate corollary. 

\begin{cor}\label{cor:Lintersecting}
Let $L$ be a set of $s$ nonnegative integers and suppose that $(\mA, \mB)$ is a pair of $L$-cross-intersecting set systems. Then
\[
	d(\mA, \mB) \le \sum_{k=0}^{s-1} \binom{n}{k}2^{n-k},
\]
with equality if and only if $L = \{0, \ldots , s-1\}$.
\end{cor}
Of course, the most trivial bound upper bound on $d(\mA, \mB)$ is given by the product $|\mA||\mB|$, and the problem of bounding $|\mA||\mB|$
for $L$-cross-intersecting $(\mA, \mB)$ has been studied before. For example, Keevash and Sudakov~\cite{crossIntersection} proved that if $L$ is a set 
of $s$ nonnegative integers and $n$ is sufficiently large (depending on $s$), then $|\mA||\mB| \le \sum_{k=0}^{s-1}\binom{n}{k}2^n$ for any $L$-cross-intersecting 
pair $(\mA, \mB)$ in $\mP[n] \times \mP[n]$, with equality if and only if $L = \{0, \ldots, s-1\}$. The example $\mA = [n]^{(\leq s - 1)}, \mB = \mP[n]$ shows that this result is tight (and also shows that \Cref{cor:Lintersecting} is tight). It is still unknown whether this bound on $|\mA||\mB|$ holds for every $s$ and $n$. More results are known when $L = \{\ell\}$ consists of a single intersection size (see~\cite{alon-lubetzky} and~\cite{crossIntersection} for more details).
 The only general upper bound (holding for all $s$, $n$) was given by Sgall~\cite{Sgall}. 
In contrast, note that our bound in \Cref{cor:Lintersecting} holds for all $s$ and $n$.

Motivated by \Cref{thm:t-avoiding-double}, it is natural to ask what happens to the parameter $d(\mA, \mB)$
when we impose that $\mA, \mB \subset [n]^{(r)}$ are both \emph{uniform}. Here it turns out that avoiding an intersection is 
not that much of a restriction, at least when $r$ is fixed and $n$ is large. Consider the following family of examples.

\begin{example}\label{example:avoiding-uniform}
For integers $r \ge 1$, $s \ge 0$ and a non-empty proper subset $X \subset [n]$ let $\mF_{X,s} = \{ F \in [n]^{(r)}: |F\cap X| \ge r-s\}$. 
For a positive integer $t \le r$ and nonnegative integers $a, b$ with $a+b\le t-1$, consider the pair
	$\left(\mF_{X, a},\, \mF_{X^c, b}\right)$.
It is easy to see that this pair is $t$-cross-avoiding (in fact, it is $\{0, \ldots , t-1\}$-cross-intersecting). Intuitively, the number of disjoint pairs should be maximized when $a =\lfloor{\frac{t}{2}}\rfloor$ and $b =\lfloor{\frac{t-1}{2}}\rfloor$ are as equal as possible.  A simple calculation shows that $d\left(\mF_{X, \lfloor{\frac{t}{2}\rfloor}},\, \mF_{X^c, \lfloor{\frac{t-1}{2}\rfloor}}\right) = \Theta_{r, t}\left(n^{2r}\right)$ and $|\mF_{X, \lfloor{\frac{t}{2}\rfloor}}||\mF_{X^c, \lfloor{\frac{t-1}{2}\rfloor}}| \sim d\left(\mF_{X, \lfloor{\frac{t}{2}\rfloor}},\, \mF_{X^c, \lfloor{\frac{t-1}{2}\rfloor}}\right)$,
when $|X| \sim cn$ for some constant $c \in (0, 1)$.
\end{example}

While we began our investigation by considering maximizing the number of disjoint pairs, this example suggests that the problem of determining the maximum number of disjoint pairs in a $t$-cross-avoiding pair $(\mA, \mB)$ of $r$-uniform set systems is roughly equivalent to determining the maximum of the product $|\mA||\mB|$ when $n$ is large and $r, t$ remain fixed. In other words, good upper bounds on $|\mA||\mB|$ translate into good upper bounds on $d(\mA, \mB)$. To formalize this we shall introduce two functions. Let 
\begin{itemize}
	\item $d(n, r, t) = \max \{d(\mA, \mB): (\mA, \mB) \subset [n]^{(r)} \times [n]^{(r)} \text{ is } t\text{-cross-avoiding} \}$,
	
	\item $p(n, r, t) = \max \{|\mA||\mB| : (\mA, \mB) \subset [n]^{(r)} \times [n]^{(r)} \text{ is } t\text{-cross-avoiding} \}$.
\end{itemize}

We prove the following theorem, which states that these two functions are asymptotically equivalent. Here, and in the sequel, we assume that $r$ and $t$ are fixed and $n \rightarrow \infty$.
\begin{thm}\label{thm:disjoint-to-product}
Let $r \ge t \ge 1$ be integers. Then
\[
	p(n, r, t) = (1 + o(1))d(n, r, t),
\]
as $n \rightarrow \infty$.
\end{thm}	
In view of \Cref{thm:disjoint-to-product}, it is perhaps more natural to provide upper bounds for the function $p(n, r, t)$ in the context of trying to obtain 
upper bounds for $d(n, r, t)$.
The function $p(n, r, t)$ has been investigated before by Frankl and R{\"o}dl~\cite{FranklRodl} when $r$ and $t$ are both linear in $n$. For single families, Ellis, Keller, and Lifshitz~\cite{ellis-keller-lifshitz} showed more recently that the maximum size of any $t$-avoiding $\mF \subset [n]^{(r)}$ with $o(n) \le r \le n/2 - o(n)$ coincides with the maximum size of a $(t+1)$-intersecting $r$-uniform family, 
answering a question of Erd{\H{o}}s and S{\'o}s~\cite{erdos-proc} for a wide range of $r$. 

 When $n \ge n_0(r, t)$, the problem of determining $p(n, r, t)$ can be viewed as the cross-analogue of a problem resolved by Frankl and F{\"u}redi~\cite{FranklFurediuniform}. They showed, in particular, that if $n$ is sufficiently large and $\mF \subset [n]^{(r)}$ is $t$-avoiding, then the family consisting of all $r$-sets containing a fixed $(t+1)$-set is optimal.
Now, note that we may assume that $t < r$, as trivially $p(n, r, r) = \frac{1}{4}\binom{n}{r}^2$.
We make 
progress in determining $p(n, r, t)$ in the first two cases, $t =1$ and $t=2$. 

\begin{restatable}{thm}{thmAvoidingUniformProduct}\label{thm:1-avoiding-uniform-product}
Let $r \geq 2$ be an integer. There exists $n_0 = n_0(r)$ such that if $n > n_0$ and $(\mA, \mB)$ is a pair of $1$-cross-avoiding $r$-uniform set systems, then
\[
	|\mA||\mB| \leq \binom{\lfloor{n/2}\rfloor}{r}\binom{\lceil{n/2}\rceil}{r}.
\]
\end{restatable}

This result is clearly tight: just consider the pair $\left(\mF_{X, 0}, \mF_{X^c, 0}\right)$ where $X \subset [n]$ has size $\lfloor{n/2}\rfloor$. It is also tight for the problem of maximizing $d(\mA, \mB)$. In other words, we have that 
$d(n, r, 1) = p(n, r, 1) = \binom{\lfloor{n/2}\rfloor}{r}\binom{\lceil{n/2}\rceil}{r}$ for $n$ sufficiently large.

Our last theorem gives an asymptotically tight upper bound for $p(n, r, 2)$.

\begin{thm}\label{thm:2-avoiding-uniform}
Suppose $r \geq 3$ and let $(\mA, \mB)$ be a pair of $2$-cross-avoiding $r$-uniform set systems.
Then
\[
	|\mA||\mB| \leq (\gamma_r + o(1))\binom{n}{r}^2,
\]
where $\gamma_r = \max_{\alpha \in [0,1]} \{\alpha^r(1-\alpha)^r + r\alpha^{r+1}(1-\alpha)^{r-1}\}$.
\end{thm}

The pair $\left(\mF_{X, 1}, \mF_{X^c, 0}\right)$ with $|X| = \alpha n$, where $\alpha \in [0, 1]$ gives the maximum value $\gamma_r$ above, shows that this upper bound is asymptotically optimal. 
Moreover, using \Cref{thm:disjoint-to-product}, we have that $p(n, r, 2) = (\gamma_r + o(1))\binom{n}{r}^2$ and 
$d(n, r, 2) = (\gamma_r + o(1))\binom{n}{r}^2$. Notice that in the case of both \Cref{thm:1-avoiding-uniform-product} and 
\Cref{thm:2-avoiding-uniform}, pairs $\left(\mF_{X, a}, \mF_{X^c, b}\right)$ where $a$ and $b$ are as equal as possible are optimal. We conjecture that this 
phenomenon persists for higher forbidden intersection sizes (see \Cref{sec:final}).

\subsection{Organization and Notation}

The remainder of this paper is organized as follows. In \Cref{sec:dp-nonuniform} we prove \Cref{thm:t-avoiding-double}, which 
implies \Cref{thm:t-avoiding-single}. In \Cref{sec:dp-uniform}, we shall prove \Cref{thm:disjoint-to-product}, \Cref{thm:1-avoiding-uniform-product}, and \Cref{thm:2-avoiding-uniform}. In the final section, we shall 
state some open problems.

Our notation is standard. For a set $X$ we let $\mP(X)$ denote the power-set of $X$ and
$X^{(r)}$ (resp., $X^{(\leq r)}$) denote the collection of all $r$-element subsets of $X$ (resp., subsets of $X$ of size 
at most $r$). We shall simply write $\mP[n]$ for $\mP([n])$. Any set system $\mF \subset X^{(r)}$ is said to be $r$-\emph{uniform} and its elements are $r$-\emph{sets}. For $\mF \subset \mP[n]$ and $T \subset [n]$ we let $\mF(T)$ denote the collection 
of sets in $\mF$ that contain $T$. When $T = \{x\}$ is a singleton we shall simply write $\mF(x)$.

\section{Disjoint pairs in $t$-cross-avoiding set systems}\label{sec:dp-nonuniform}

Our aim in this section is to establish \Cref{thm:t-avoiding-double}, which we restate for convenience.
\thmtavoidingdouble*

Let us point out one fact before giving a proof of the theorem.
Note that if we let $f(n, t) = \sum_{k=0}^{t-1} \binom{n}{k}2^{n-k}$, then $f$ satisfies the recurrence
\[
	f(n, t) = 2f(n-1, t) + f(n-1, t-1), 
\]
for natural numbers $n, t \geq 1$.

\begin{proof}
We shall apply induction on $n$ and $t$. The base case $t = 0$ holds trivially for every value of $n$.
Therefore, we fix $t >0$ and assume the theorem holds for $t' < t$ (and every value of $n$), and we may suppose 
the theorem holds for $t' = t$ and all $n' < n$. We aim to show it holds for $t' = t$ and $n' = n$.

To do so, suppose that $(\mA, \mB) \subset \mP[n] \times \mP[n]$ is $t$-cross-avoiding. We shall split $\mA$ and $\mB$ into certain subfamilies. More specifically, let 
$\mA_n =\{A\in\mA: n \in A\}$ and $\mA_{0}=\{A\in\mA: n\not\in \mA\}$, and define $\mB_n$ and $\mB_0$ analogously. We further identify three subfamilies of $\mA_n$, namely, 
\begin{itemize}

\item $\mA^{*}_n=\{A\in\mA_n: A\setminus \{n\}\in \mA\}$,
\item $\mA_n^{t+1}=\{A\in\mA_n: \exists B\in \mB_n \text{ with } |A\cap B|=t+1\}$, and 
\item $\mathcal{X}=\mA_n \setminus (\mA^{*}_n \cup \mA_n^{t+1})$. 
\end{itemize}


We define similarly the corresponding subfamilies $\mB^*_n ,\mB^{t+1}_n$, and $\mathcal{Y} = \mB_n\setminus (\mB_n^* \cup \mB_n^{t+1})$ of $\mB_n$.
Note that the subfamilies defined above actually partition $\mA_n$ and $\mB_n$. Indeed, suppose $A \in \mA^*_n \cap \mA_n^{t+1}$. Then there exists $B \in \mB_n$ 
such that $|A \cap B| = t+1$. But we also have that $A \setminus \{n\} \in \mA$ and then $|A \setminus \{n\} \cap B| = t$, a contradiction. The same argument shows that 
$\mB_n^*$ and $\mB_n^{t+1}$ are disjoint.

For a subset $A \subset [n]$, a family $\mathcal{F} \subset \mP[n]$, and $i \in [n]$ let $D_i(A) = A\setminus \{i\}$ and 
\[
	D_i(\mathcal{F}) = \{ D_i(A): A \in \mathcal{F} \}.
\]
To reduce clutter we shall simply write $D$ for $D_n$. Our aim is to apply $D$ to suitable pairs of families and apply induction. 
Indeed, consider the pairs
\[
(\mA_0 \cup D(\mathcal{X}\cup \mA_n^{t+1}), \mB_0 \cup D(\mathcal{Y})),
\]
 and 
 \[
(\mA_0 \cup D(\mathcal{X}), \mB_0 \cup D(\mathcal{Y}\cup \mB_n^{t+1})). 
\]
Of course, each of the families in these pairs belongs to $\mP[n-1]$. We also need that the above pairs are $t$-cross-avoiding, which we formulate in the following claim.
\begin{claim}\label{claim:t-avoiding}
$(\mA_0 \cup D(\mathcal{X}\cup \mA_n^{t+1}), \mB_0 \cup D(\mathcal{Y}))$ and $(\mA_0 \cup D(\mathcal{X}), \mB_0 \cup D(\mathcal{Y}\cup \mB_n^{t+1}))$ are $t$-cross-avoiding pairs of set systems.
\end{claim}
\begin{proof}
We only prove that the first pair is $t$-cross-avoiding. The second follows by a similar argument. By way of contradiction, 
suppose there exists $A\in \mA_0 \cup D(\mathcal{X}\cup \mA_n^{t+1})$ and $B \in \mB_0 \cup D(\mathcal{Y})$ such that $|A\cap B|= t$.  Clearly, either $B\in \mB$ or $B\cup \{n\} \in \mB$. If $A\in \mA_{0}$, then $|A\cap B|=|A\cap(B\cup \{n\})|=t$, which is a contradiction. So we may assume that $A\cup \{n\}\in \mA$ and similarly $B\cup \{n\}\in \mB$. Hence $|(A\cup \{n\})\cap (B\cup \{n\})|=t+1$ which would imply $B\cup \{n\}\in \mB_n^{t+1}$, which is again a contradiction. This completes the proof. 
\end{proof}

Our second claim exhibits a pair of subfamilies that are, in fact, $(t-1)$-cross-avoiding.
\begin{claim}\label{claim:t-1avoiding}
The pair of set systems $(D(\mA^{*}_n) ,D(\mB^{*}_n))$ is $(t-1)$-cross-avoiding in $\mP[n-1] \times \mP[n-1]$.
\end{claim}
\begin{proof}
Indeed, suppose there is $A \in D(\mA^*_n)$ and $B \in D(\mB^*_n)$ such that $|A \cap B| = t -1$. But since $A' = A\cup \{n\} \in \mA_n$ and $B'= B \cup \{n\} \in \mB_n$, 
we have that $|A'\cap B'| = t$, a contradiction.
\end{proof}

We shall now count the disjoint pairs $(A, B)$ with $A \in \mA$ and $B \in \mB$ in such a way that every such pair gets counted except those disjoint pairs in $(D(\mA^*_n), D(\mB^*_n))$. The following lemma summarizes this, from which our theorem follows easily. Before stating it we shall rename some families in order to make the statement cleaner. 
Let
\begin{itemize}
\item $(\mA_0 \cup D(\mathcal{X}\cup \mA_n^{t+1}), \mB_0 \cup D(\mathcal{Y})) = (\mF_1, \mF_2)$, and
\item $(\mA_0 \cup D(\mathcal{X}), \mB_0 \cup D(\mathcal{Y}\cup \mB_n^{t+1})) = (\mF_3, \mF_4)$.
\end{itemize}
With this in mind we shall prove the following.
 \begin{lem}\label{lem:counting}
$d(\mF_1, \mF_2)+d(\mF_3, \mF_4) \ge d(\mA,\mB)-d(D(\mA^{*}_n),D(\mB^{*}_n))$.
\end{lem}
\begin{proof}
Let us see how the left-hand side $d(\mF_1, \mF_2) + d(\mF_3, \mF_4)$ counts disjoint pairs. Note that it counts every disjoint pair in $(\mA_n^{t+1} \cup \mX, \mB_0)$ and 
$(\mA_0, \mB_n^{t+1}\cup \mY)$ once (it may count more; namely, disjoint pairs in $(D(\mX), D(\mY))$ that do not exist in $(\mA, \mB)$). Furthermore, it counts disjoint pairs in 
$(\mA_0, \mB_0)$ twice. Such pairs between $\mA_0$ and $\mB_0$ can be broken up into the following three types:
\begin{itemize}
\item those in $(D(\mA^*_n), D(\mB_n^*))$;
\item those in $(D(\mA^*_n), \mB_0 \setminus D(\mB_n^*))$;
\item those in $(\mA_0\setminus D(\mA^*_n), D(\mB_n^*))$.

\end{itemize}
The remaining disjoint pairs to be counted are those in $(\mA_n^*, \mB_0)$ and $(\mA_0, \mB_n^*)$. Since 
\[
d(D(\mA^*_n), \mB_0 \setminus D(\mB_n^*)) = d(\mA_n^*, \mB_0\setminus D(\mB_n^*)),
\]
and, similarly, $d(\mA_0\setminus D(\mA^*_n), D(\mB_n^*)) = d(\mA_0\setminus D(\mA^*_n), \mB_n^*)$, we have that the disjoint pairs in $(\mA_n^*, \mB_0\setminus D(\mB_n^*))$ and
$(\mA_0\setminus D(\mA^*_n), \mB_n^*)$ get counted when we count those disjoint pairs in $(\mA_0, \mB_0)$. Furthermore, since $d(D(\mA^*_n), D(\mB^*_n)) = d(\mA_n^*, D(\mB^*_n))$, the disjoint pairs in $(\mA_n^*, D(\mB^*_n))$ also get counted whenever we count pairs in $(\mA_0, \mB_0)$. As $d(\mF_1, \mF_2) + d(\mF_3, \mF_4)$ counts the disjoint pairs in $(\mA_0, \mB_0)$ twice we can equivalently say that it counts 
\begin{itemize}
\item disjoint pairs in $(\mA_0, \mB_0)$ once;
\item disjoint pairs in $(\mA_0\setminus D(\mA_n^*), \mB_n^*)$ once;
\item disjoint pairs in $(\mA_n^*, \mB_0\setminus D(\mB_n^*))$ once;
\item disjoint pairs in $(\mA_n^*, D(\mB^*_n))$ once.

\end{itemize}
Thus the only disjoint pairs in $(\mA, \mB)$ not counted by $d(\mF_1, \mF_2) +  d(\mF_3, \mF_4)$ are those in $(D(\mA^*_n), \mB_n^*)$, and since $d(D(\mA^*_n), \mB_n^*) = d(D(\mA^*_n), D(\mB^*_n))$, we have that 
\[
	d(\mF_1, \mF_2)+d(\mF_3, \mF_4) \ge d(\mA,\mB)-d(D(\mA^{*}_n),D(\mB^{*}_n)),
\]
as claimed.
\end{proof}

\Cref{thm:t-avoiding-double} now follows easily from \Cref{lem:counting}. Indeed, by \Cref{claim:t-avoiding}, $(\mF_1, \mF_2)$ and $(\mF_3, \mF_4)$ are both $t$-cross-avoiding in 
$\mP[n-1] \times \mP[n-1]$, so by induction we have $d(\mF_1, \mF_2) \le f(n-1, t)$ and $d(\mF_3, \mF_4) \le f(n-1, t)$. By \Cref{claim:t-1avoiding}, $(D(\mA^*_n), D(\mB^*_n)) \subset \mP[n-1] \times \mP[n-1]$ is $(t-1)$-cross-avoiding, and so $d(D(\mA^*_n), D(\mB^*_n)) \le f(n-1, t-1)$. Therefore, by \Cref{lem:counting} and using the recurrence for $f$, we have
\[
	d(\mA, \mB) \le 2f(n-1, t) + f(n-1, t-1) = f(n, t), 
\]
as claimed.
\end{proof}

\subsection{Characterization of extremal examples}\label{subsec:classification}
To end this section, let us characterize the extremal examples occurring in \Cref{thm:t-avoiding-double}. We must break the analysis up into two cases, when $t = 1$ and when $t > 1$, as the extremal behaviour is different. We consider first the case $t > 1$.

\begin{itemize}
\item $ t >1$
\end{itemize}
Observe that when $n=t$ equality is trivially only attained when the families are $(\mP[n] \setminus \{[n]\},\mP[n])=([n]^{(\leq n-1)},\mP[n])$. We may assume now that $n > t$. 
From the proof of \Cref{thm:t-avoiding-double}, both pairs $(\mF_1,\mF_2)$ and $(\mF_3,\mF_4)$ must satisfy $d(\mF_1,\mF_2)=d(\mF_3,\mF_4)=f(n-1,t)$.
By induction on $n$, we may assume without loss of generality that $\mA_0 \cup D(\mathcal{X}\cup \mA_n^{t+1})= \mP[n-1]$ and $\mB_0 \cup D(\mY)=[n-1]^{(\leq t-1)}$. We will show that $\mA = \mP[n]$ and $\mB = [n]^{(\leq t-1)}$. Since $\varnothing \in \mA_0$ (as $t\geq 1$) and, by the definition of $\mY$, for any element $B \in \mY$, $B \setminus \{n\}$ can be added to $\mB_0$ implying that $\mY$ is empty. 
We then have that $\mB= \mB_0 \cup \mB_n^{*} \cup \mB_n^{t+1}$ and $\mB_0=[n-1]^{(\leq t-1)}$. Similarly we must have that $\mX$ is empty and so $\mA=\mA_0 \cup \mA_n^{*} \cup \mA_n^{t+1}$ and $\mA_0 \cup D(\mA_n^{t+1})=\mP[n-1]$. 
Moreover, we must have that $d(\mF_3,\mF_4)=d(\mA_0,\mB_0\cup D(\mB_n^{t+1}))=f(n-1,t)$ and again by induction, either $\mA_0=\mP[n-1]$ and $\mB_0\cup D(\mB_n^{t+1})=[n-1]^{(\leq t-1)}$ or $\mA_0=[n-1]^{(\leq t-1)}$ and $\mB_0\cup D(\mB_n^{t+1})=\mP[n-1]$. We split our analysis into two parts according to whether the former or latter case holds. 

$(i)$. Suppose the latter case holds. Then any set $A \in [n-1]^{(t)}$ must be of the form $D(A')$ for some $A' \in \mA_n^{t+1}$ and similarly of the form $D(B')$ for some $B'\in \mB_n^{t+1}$. If $n>t+1$ then we reach an immediate contradiction as we can find two elements $A,B \in [n-1]^{(t)}$ with $|A\cap B|=t-1$ which would imply $|(A\cup \{n\}) \cap (B\cup \{n\})|=t$. So suppose that $n=t+1$. Now, neither $\mA_n^{t+1}$ nor $\mB_n^{t+1}$ can be empty. For if $\mA_n^{t+1} = \varnothing$, then $\mP[t] = \mA_0 = [t]^{(\le t-1)}$, which is a contradiction. Similarly, $\mB_n^{t+1} \neq \varnothing$. Then $\mA_n^{t+1}=\mB_n^{t+1}=\{[n]\}$, and so no set $A\in [n]^{(t)}$ can belong to either $\mA$ or $\mB$. It follows that the only sets that can belong to $\mA_n^*$ are of the form $A \cup \{n\}$ for some $A \in \mA_0$ with $|A| \le t-2$ (and similarly for the sets in $\mB_n^*$). Accordingly, $\mA=\mB=[n]^{(\leq t-1)}\cup \{[n]\}$, which is impossible as the number of disjoint pairs is smaller than $f(n,t)$.

$(ii)$. Suppose the former case holds, that $\mA_0=\mP[n-1]$ and $\mB_0\cup D(\mB_n^{t+1})=[n-1]^{(\leq t-1)}$. Since we cannot have an element $A\in \mA_n^{t+1}$ and $D(A)\in A_0$, we must have $A_n^{t+1}=\varnothing$ and, analogously, $B_n^{t+1}=\varnothing$. It follows that $\mA = \mA_n^* \cup \mP[n-1]$ and $\mB = \mB_n^* \cup [n-1]^{(\le t-1)}$. Let us first deal with the case $t = 2$. If $\mB_n^*$ contains no sets of the form $\{i, n\}$, then $\mB_n^* = \{\{n\}\}$ and we are done. Our aim is to show that if $\mB_n^*$ contains a $2$-set, then the number of disjoint pairs is strictly smaller than $f(n, 2)$. So suppose, by way of contradiction, that $\mB_n^*$ contains sets $\{i_1, n\}, \ldots , \{i_l, n\}$ for some $i_1, \ldots, i_l \in [n-1]$. It follows that $\mA_n^*$ can consist of only sets containing $n$ and avoiding $i_1, \ldots , i_l$. Therefore, we may assume $|\mA_n^*| = 2^{n-1-l}$. The number of disjoint pairs between $\mP[n-1]$ and $\mB$ is $2^{n-1} + (n-1)2^{n-2} + 2^{n-1} + l2^{n-2} = 2^n + (n-1+l)2^{n-2}$. The number of disjoint pairs between $\mA_n^*$ and $\mB$ is $(l+1)2^{n-1-l} + (n-1-l)2^{n-2-l}$. Since $f(n, 2) = 2^n + n2^{n-1}$ we have to check that
\begin{equation}\label{eq:inequality}
(n-1+l)2^{n-2} + (l+1)2^{n-1-l} + (n-1-l)2^{n-2-l} < n2^{n-1},
\end{equation}
for $1 \le l \le n-1$. It is easy to check that~(\ref{eq:inequality}) holds for $l = 1, 2$ (bearing in mind that we may assume $n >2$). Further,~(\ref{eq:inequality}) is equivalent to 
$n > \frac{2^l(l-1) + l + 1}{2^l -1}$, which is true since $\frac{2^l(l-1) + l + 1}{2^l -1} \le l$ for $l \ge 3$, and also since $l < n$.
Accordingly, $\mB_n^*$ contains no $2$-sets, and so the proof is complete for $t=2$.

Finally, we see in the proof of \Cref{thm:t-avoiding-double} that in order to have equality, it must hold that $d(D(\mA^*_n), D(\mB^*_n)) = f(n-1, t-1)$. By induction on $n$ and $t$ ($t =2$ being the base case), we have that $D(\mA^*_n) = \mP[n-1]$ and $D(\mB^*_n) = [n-1]^{(\le t-2)}$. So, since $\mA = \mA_n^* \cup \mP[n-1]$ and $\mB = \mB_n^* \cup [n-1]^{(\le t-1)}$, it follows that 
$\mA = \mP[n]$ and $\mB = [n]^{(\le t-1)}$ as required.

\begin{itemize}
\item $t =1$
\end{itemize}
We claim that equality holds only if $\mA = \mP(S), \mB = \mP([n]\setminus S)$ (or vice-versa) for some $S \subseteq [n]$.
This is certainly true for $n=1$. Let $n \geq 2$ and suppose the result holds for smaller values of $n$. 
As before, since both pairs $(\mF_1,\mF_2)$ and $(\mF_3,\mF_4)$ must satisfy $d(\mF_1,\mF_2)=d(\mF_3,\mF_4)=f(n-1,1)$, by induction on $n$, we may assume $\mF_1= \mA_0 \cup D(\mathcal{X}\cup \mA_n^{2})= \mP(W)$ for some $W \subseteq [n-1]$ and $\mF_2=\mB_0 \cup D(\mY)=\mP([n-1]\setminus W)$. Similarly  $\mF_3= \mA_0 \cup D(\mathcal{X})= \mP(W^{\prime})$ and $ \mF_4= \mB_0 \cup D(\mathcal{Y}\cup \mB_n^{2})= \mP([n-1]\setminus W^{\prime})$.
Note that as before we may assume $\mathcal{X}$ and $\mathcal{Y}$ are empty.  

Clearly we have that $W^{\prime}\subseteq W$ and we shall show they actually must be equal. Suppose first that $|W\setminus W^{\prime}|\geq 2$ and let $i_1,i_2$ be two distint elements in $W\setminus W^{\prime}$. By definition, the sets $\{i_1\}, \{i_2\}$ belong to $\mA_0\cup D(\mA^2_{n})$ and to
$\mB_0\cup D(\mB^2_{n})$. But this implies both $\{i_1,n\},\{i_2,n\}$ belong to $\mA^2_{n}$ and to $\mB^2_{n}$, which is a contradiction since we generate a cross-intersection of size $1$. So we may assume that $W\setminus W^{\prime}=\{i\}$, which implies $\{i,n\}$ belongs to $\mA^2_{n}$ and to $\mB^2_{n}$.  Note that both $\mA_n^{*},\mB_n^{*}$ are empty. Indeed, for any element $A \in \mA_n^{*}$ (or $\mB_n^{*}$), the set $A\setminus\{n\}$ belongs to $\mA_0$ (or $\mB_0$) and therefore $A\setminus\{n\}\subseteq W^{\prime}$ (or $ A\setminus\{n\}\subseteq [n-1]\setminus W$). In any case, $A\cap \{i,n\}=\{n\}$, which is impossible. 
We must then have that $\mA= \mP(W^{\prime})\cup (\{i,n\} \vee \mP(W^{\prime}))$ and $\mB=\mP([n-1]\setminus W)\cup (\{i,n\} \vee \mP([n-1] \setminus W))$, for some $W\in \mP([n-1])$ and $i \in [n-1]$ with $W^{\prime}=W\setminus \{i\}$ (as usual, for a set $A$ and a family $\mF$, $A \vee \mF := \{A\cup F: F \in \mF\}$). A simple calculation shows there are exactly $2^{n-2}+2^{n-1}<2^{n}$ disjoint pairs in $(\mA,\mB)$, a contradiction. 
It follows that $W=W^{\prime}$. Hence $\mA_n^{2}$ and $\mB_n^{2}$ must be empty.  Clearly at most one of the sets $\mA_n^{*}, \mB_n^{*}$ can be non-empty, and our result follows.

\section{Disjoint pairs in uniform set systems}\label{sec:dp-uniform}

Our aim in this section is prove Theorems~\ref{thm:disjoint-to-product},~\ref{thm:1-avoiding-uniform-product}, and~\ref{thm:2-avoiding-uniform}. 
We first prove \Cref{thm:disjoint-to-product} which provides a relation between the maximum number of disjoint pairs and the maximum size of the product of two $t$-cross-avoiding $r$-uniform set systems. Recall that, for positive integers $t \le r$ we have defined $d(n, r, t)$ to be the maximum of $d(\mA, \mB)$ over all $t$-cross-avoiding $r$-uniform $(\mA, \mB)$ on the ground set $[n]$. Analogously, we have defined $p(n, r, t)$ to be the maximum of the product $|\mA||\mB|$ over all such pairs of set systems. To these two functions we add a third:
\[
	p^*(n, r, t) := \max \{|\mA||\mB|: (\mA, \mB) \subset [n]^{(r)} \times [n]^{(r)} \text{ is } \{0, \ldots , t-1\}\text{-cross-intersecting} \}.
\]
	
Clearly, $p^*(n, r, t) \le p(n, r, t)$.	
In order to prove \Cref{thm:disjoint-to-product}, we first show that $p(n, r, t) \sim p^*(n, r, t)$ as $n \rightarrow \infty$.
First, let us recall a notion that will be useful in the proof. Let $\mF$ be a family of subsets of $[n]$. A \emph{delta-system} in $\mF$ of size $s$ with core $C$ is a collection 
of sets $F_1, \ldots , F_s \in \mF$ such that for every $i \neq j$, $F_i \cap F_j = \cap_{k=1}^{s} F_k = C$. The following lemma upper bounds $p(n, r, t)$ in terms of $p^*(n, r, t)$. As \Cref{example:avoiding-uniform} shows that $d(n, r, t) = \Omega_{r, t}\left(n^{2r}\right)$, the second term in the upper bound is negligible, so we can establish \Cref{thm:disjoint-to-product} by proving an upper bound of the form $p^*(n, r, t) \le d(n, r, t) + O_{r, t}\left(n^{2r-1}\right)$.

\begin{lem}\label{lem:technical1}
Let $t, r$ be positive integers with $t \le r$. Then
\[
	p(n, r, t) \leq p^*(n, r, t) + C_{r, t}n^{2r-1},
\]
for some constant $C_{r, t}$ depending on $r$ and $t$.
\end{lem}

\begin{proof}
Let $(\mA, \mB)$ be a $t$-avoiding pair of $r$-uniform families with $|\mA||\mB| = p(n, r, t)$. We say that a $t$-set $T \subset [n]$ is $\mA$-\emph{good} (resp., $\mB$-\emph{good}) if there exists a delta-system in $\mA$ (resp., $\mB$) of size at least $r-t+1$ with core $T$. Observe that if $T$ is $\mA$-good, then no set in $\mB$ contains $T$ (the symmetric claim holds if $T$ is $\mB$-good). 
Indeed, suppose otherwise that some $B \in \mB$ contains $T$. Let $\Delta \subset \mA$ be the corresponding delta-system 
with core $T$, so that $|\Delta| \geq r-t+1$. Then $B \setminus T$ has size $r-t$ and accordingly there exists $A \in \Delta$ 
such that 
\[
	(A\setminus T) \cap (B\setminus T) = \varnothing.
\]
It follows that $|A\cap B| = t$, a contradiction.

Let $\mT$ be the collection of $t$-sets which are neither $\mA$-good nor $\mB$-good and let 
\[
	\mA_0= \bigcup_{T \in \mT} \mA(T) \text{ \:\: and \:\: } \mB_0 = \bigcup_{T \in \mT} \mB(T).
\]

We claim that the subfamilies $\mA_0$ and $\mB_0$ are small. Indeed, suppose $T \in \mT$. Then any maximum-sized 
delta-system $\Delta \subset \mA$ has size $|\Delta| \leq r-t$. It follows that any set in $\mA(T)$ must non-trivially intersect a set in $\Delta$ outside of $T$. Therefore, it is easy to see that 
\[
	|\mA(T)| \leq 2^{(r-t)^2} \binom{n}{r-t-1}, 
\]
and the same bound holds for $|\mB(T)|$. Accordingly, $|\mA_0|, |\mB_0| \leq c_{r, t} n^{r-1}$ for some constant $c_{r, t}$, 
depending only on $r$ and $t$. Now, let
\[
	\mA' = \mA \setminus \mA_0 \text{ \:\: and \:\:} \mB' = \mB \setminus \mB_0,
\]
and note that the pair $(\mA', \mB')$ is $\{0, \ldots , t-1\}$-intersecting, for if $A' \in \mA'$ and $B' \in \mB'$ intersect 
in $t$ points, then this $t$-set is both $\mA$-good and $\mB$-good, which is impossible. Finally, we see that 
\begin{align*}
p^*(n, r, t) \ge |\mA'||\mB'| &= \left(|\mA| - |\mA_0|\right)\left(|\mB| - |\mB_0|\right) \\
& \ge |\mA||\mB| - O_{r, t}\left(n^{2r-1}\right)\\
&= p(n, r, t) - O_{r, t}\left(n^{2r-1}\right), 
\end{align*}
completing the proof.
\end{proof}

With \Cref{lem:technical1} in mind we can now complete the proof of \Cref{thm:disjoint-to-product}, which asserts that the functions 
$p(n, r, t)$ and $d(n, r, t)$ are essentially equivalent as $n \rightarrow \infty$.

\begin{proof}[ of \Cref{thm:disjoint-to-product}]
First note that $p^*(n, r, t) \le d(n, r, t) + C_{r, t}n^{2r-1}$ for some constant $C_{r, t}$ depending on $r, t$. Indeed, 
if $(\mA, \mB)$ is $\{0, \ldots , t-1\}$-cross-intersecting with $|\mA||\mB| = p^*(n, r, t)$, then we can count 
\[
	|\mA||\mB| = d(\mA, \mB) + \sum \limits_{\substack{A \in \mA, B \in \mB \\ A\cap B \neq \varnothing}}1.
\]
Now, for each element $A \in \mA$ there are at most $2^{t}\binom{n-r}{r-1}$ sets in $[n]^{(r)}$ which have 
non-empty intersection with $A$. Hence, the second summand on the right-hand side is bounded by 
$|\mA|2^t\binom{n-r}{r-1} \le 2^t\binom{n}{r}\binom{n-r}{r-1} \le C_{r, t}n^{2r-1}$.

Now, applying \Cref{lem:technical1} we see that 
\[
	p(n, r, t) \leq d(n, r, t) + c_{r, t}n^{2r-1}, 
\]
for some constant $c_{r, t}$ depending on $r, t$. \Cref{example:avoiding-uniform} shows that $d(n, r, t) = \Omega_{r, t}\left(n^{2r}\right)$, 
and so the result holds as claimed.
\end{proof}

 In the next two subsections 
we shall shift our focus to proving upper bounds for $p(n, r, t)$ in the first two cases $t = 1, 2$. When $t = 1$, the extremal example exhibits some symmetry (in particular, both 
families have the same size). This 
symmetry disappears when $t =2$, indicating that the problem of bounding $p(n, r, t)$ for general $t$ could be quite challenging.

\subsection{Forbidding an intersection of size $1$}

It is very easy to give an upper bound for $p^*(n, r, 1)$, and so, by \Cref{lem:technical1}, this 
translates to an asymptotic upper bound for $p(n, r, 1)$. Indeed, if $\mA, \mB \subset [n]^{(r)}$ are $\{0\}$-cross-intersecting, then rather trivially 
$\left(\bigcup_{A \in \mA}A\right)\cap \left(\bigcup_{B \in \mB} B\right) = \varnothing$, so we may assume that $\mA = X^{(r)}$ and 
$\mB = \left( [n]\setminus X\right)^{(r)}$ for some set $X \subset [n]$. If $|X| = x$, then we have 
\[
	|\mA||\mB| = \binom{x}{r}\binom{n-x}{r}, 
\]
and the right-hand side is maximized when $x = \lfloor{\frac{n}{2}}\rfloor$. Hence,
\[
p(n, r, 1) = (1+o(1))\binom{\lfloor{n/2}\rfloor}{r}\binom{\lceil{n/2}\rceil}{r}.
\]
However, in this case we are able to remove the error term and prove 
an exact upper bound, for $n$ sufficiently large.

\thmAvoidingUniformProduct*

\begin{proof}

Suppose that $(\mA, \mB)$ is $1$-cross-avoiding and maximize $|\mA||\mB|$, and suppose without loss of generality 
that $|\mA| \geq \binom{\lceil{n/2}\rceil}{r}$. As in the proof of \Cref{lem:technical1}, we give a reduction via delta-systems.
More precisely, recall that we say $x \in [n]$ is $\mA$-\emph{good} (resp., $\mB$-\emph{good}) if there exists a delta-system in $\mA$ (resp., $\mB$)
of size at least $r$ with core $\{x\}$. Let $X$ and $Y$ denote the set of $\mA$-good and $\mB$-good points,
respectively, and observe that $A\cap Y = \varnothing$ for every $A \in \mA$ and $B \cap X = \varnothing$ for every 
$B \in \mB$. We therefore obtain a 
partition $[n] = X \cup Y \cup Z$ where $Z$ denotes the set of points which are neither $\mA$-good nor $\mB$-good. We may 
also assume that $X^{(r)} \subset \mA$ and $Y^{(r)} \subset \mB$, since adding any $r$-set contained in $X$ to $\mA$ and adding any $r$-set contained in $Y$ to $\mB$ does not violate the $1$-cross-avoiding property. It also follows from the proof of \Cref{lem:technical1} that, if 
$\mA_0 := \{A \in \mA: A\cap Z \neq \varnothing \}$ and 
$\mB_0 := \{B \in \mB: B\cap Z \neq \varnothing \}$, then $|\mA_0|, |\mB_0| \leq \frac{2^{(r-1)^2}}{(r-2)!}n^{r-1}$. 
Let us write $|\mA| = \binom{x}{r} + |\mA_0|$ and $|\mB| = \binom{y}{r} + |\mB_0|$, where $x := |X|$ and $y := |Y|$, so
\[
	|\mA||\mB| = \binom{x}{r}\binom{y}{r} + |\mA_0|\binom{y}{r} + |\mB_0|\binom{x}{r} + |\mA_0||\mB_0|.
\]

The rest of the proof will be broken into two claims. The first claim asserts that we may assume that the size of $Y$ is large (i.e., linear in $n$). The second claim states that, under the 
assumption that $\mA, \mB$ maximize $|\mA||\mB|$, no point of $[n]$ can be neither $\mA$-good nor $\mB$-good. We therefore obtain the structural information that $\mA = X^{(r)}$ and $\mB = Y^{(r)}$.

\begin{claim}\label{claim:Ybig}
We may assume that $y \geq \beta n$, where $\beta = \beta(r) = \frac{(r!)^{2/r}}{200^{1/r}4r^2}$ (as long as $n$ is sufficiently large).
\end{claim}

\begin{proof}
Put $c_r = \frac{2^{(r-1)^2}}{(r-2)!}$, let $\beta$ be as above, and suppose that 
$y < \beta n$. Using the fact that $|\mA_0|, |\mB_0| \leq c_rn^{r-1}$ and crudely bounding $\binom{x}{r} \leq \binom{n}{r}$, 
we have that
\begin{align*}
|\mA||\mB| &\leq \binom{n}{r}\binom{\beta n}{r} + 2c_rn^{2r-1} + c_r^2n^{2r-2} \\
&\leq \frac{\beta^r}{(r!)^2}n^{2r} + 3c_rn^{2r-1},
\end{align*}
where in the first line we have used the monotonicity of the function $z \mapsto \binom{z}{r}$ (for $z \ge r - 1$) and the inequality 
$\binom{\theta n}{r} \leq \theta^r\binom{n}{r}$, valid for any $ \theta \in (0,1)$ with $\theta n > r$. Assuming that $n \geq 600c_r4^rr^{2r}$ 
we have that $3c_r/n \leq 1/2004^rr^{2r}$, and therefore by our assumption on $\beta$
\[
	|\mA||\mB| < \frac{1}{100}\frac{n^{2r}}{4^rr^{2r}} \leq \frac{1}{100}\binom{n/2}{r}^2 \leq \binom{\lfloor{n/2}\rfloor}{r}\binom{\lceil{n/2}\rceil}{r}, 
\]
completing the proof of \Cref{claim:Ybig}.
\end{proof}

The proof of \Cref{thm:1-avoiding-uniform-product} will be nearly finished once we establish that the set $Z$ of points which are neither $\mA$-good nor $\mB$-good is empty. Our 
second claim asserts just this.

\begin{claim}\label{claim:Zempty}
$Z = \varnothing$.
\end{claim}
\begin{proof}

Suppose to the contrary that there is some $x \in Z$. 
Form a new pair $(\mA', \mB')$ of $1$-cross-avoiding families in the following way.
First, create $\mA'$ by removing all sets of $\mA$ that contain $x$. We are then free to add to $\mB$ all sets of the form
$B \cup \{x\}$ where $B \subset Y$ is a subset of size $r-1$ (note that as long as $\mA(x) \neq \varnothing$, none of these sets 
originally belonged to $\mB$ 
as otherwise there would be a cross-intersection of size $1$).
It follows that $(\mA', \mB')$ is $1$-cross-avoiding and 
\begin{align*}
|\mA'||\mB'| &= \left(|\mA| - |\mA(x)|\right)\left(|\mB| + \binom{y}{r-1}\right) \\
&= |\mA||\mB| + |\mA|\binom{y}{r-1} - |\mA(x)||\mB| - |\mA(x)|\binom{y}{r-1},\\
\end{align*}
so if $|\mA| > \left(|\mB|\binom{y}{r-1}^{-1} + 1\right) |\mA(x)|$, then we reach a contradiction to the 
maximality of $|\mA||\mB|$. But by \Cref{claim:Ybig} we have $\binom{y}{r-1} \geq \frac{\beta^{r-1}}{(r-1)^{r-1}}n^{r-1}$
and so the right-hand side is at most 
\begin{equation}
	c_rn^{r-2}\left(1+ \frac{c_r(r-1)^{r-1}}{\beta^{r-1}r!}n\right) \leq \frac{2c_r^2(r-1)^{r-1}}{\beta^{r-1}r!}n^{r-1} \label{eqn:eq1}
\end{equation}

Now, as long as $n > \frac{2^{r+1}c_r^2r^r(r-1)^{r-1}}{\beta^{r-1}r!}$, the right-hand side of (\ref{eqn:eq1}) is strictly less than 
\[
	\frac{n^r}{2^rr^r} \leq \binom{n/2}{r} \leq \binom{\lceil{n/2}\rceil}{r} \leq |\mA|, 
\]
and the proof of \Cref{claim:Zempty} is complete.
\end{proof}

Since $Z = \varnothing$ it follows that $\mA = X^{(r)}$ and $\mB = Y^{(r)}$. Accordingly, $|\mA||\mB| = 
\binom{x}{r}\binom{n-x}{r}$, which is maximized when $x = \lfloor{\frac{n}{2}}\rfloor$. 
\Cref{thm:1-avoiding-uniform-product} therefore holds with $n_0(r) = \max \{ 600c_r4^rr^{2r} , \frac{2^{r+1}c_r^2r^r(r-1)^{r-1}}{\beta^{r-1}r!} \} $.
\end{proof}

\subsection{Forbidding an intersection of size $2$}

The extremal example showing that \Cref{thm:1-avoiding-uniform-product} is tight is symmetric in the sense that both families in the pair have the same size. We shall see now that this kind of symmetry is lost when forbidding a cross-intersection of size $2$. However, in view of our reduction via \Cref{lem:technical1}, \Cref{thm:2-avoiding-uniform} will follow quite easily from a result of Huang, Linial, Naves, Peled 
and Sudakov~\cite{HuangLinialNavesPeledSudakov}, and independently in a weaker form by Frankl, Kato, Katona and Tokushige~\cite{FranklKatoKatonaTokushige}. In order to state this result we need to introduce some notation. Following the first set of authors, for a $k$-vertex graph $H$ and an $n$-vertex graph $G$ 
let $\text{Ind}(H; G)$ denote the collection of induced copies of $H$ in $G$. The \emph{induced} $H$-\emph{density} in $G$ is defined as 
\[
	d(H; G) = \frac{|\text{Ind}(H;G)|}{\binom{n}{k}}.
\]

\begin{thm}\label{thm:densities}
Let $r, s \ge 2$ be integers and suppose that $d\left(\overline{K}_r; G\right) \ge p$ where 
$G$ is an $n$-vertex graph and $0\le p \le 1$. Let $q$ be the unique root of $q^r + rq^{r-1}(1-q) = p$ 
in $[0, 1]$. Then $d\left(K_s; G\right) \le M_{r, s, p} + o(1)$, where
\[
	M_{r, s, p} := \max\{ (1-p^{1/r})^s + sp^{1/r}(1-p^{1/r})^{s-1}, (1-q)^s\}.
\]

\end{thm}
After these preparations, \Cref{thm:2-avoiding-uniform} is easily proved.
\begin{proof}[ of \Cref{thm:2-avoiding-uniform}]
Let $(\mA, \mB)$ be a pair of $r$-uniform families. By \Cref{lem:technical1}, we may assume that $(\mA, \mB)$ is 
$\{0, 1\}$-cross-intersecting. Thus, the pair $(\mA, \mB)$ gives rise to a red-blue colouring of the edges of $K_n$ such that 
every $r$-set in $\mA$ induces a red copy of $K_r$, and every $r$-set in $\mB$ induces a blue copy of $K_r$. We may assume 
that $|\mA| = \alpha^r\binom{n}{r}$ for some $\alpha \in (0, 1)$. Then in \Cref{thm:densities} we may take $G = K_n$, $r = s$, and 
$p = \alpha^r$. It follows that 
\[
	|\mB| \le (M_{r, r, \alpha^r} + o(1))\binom{n}{r},
\]
and hence
\[
	|\mA||\mB| \leq (\gamma_r + o(1))\binom{n}{r}^2,
\]
where $\gamma_r = \max_{\alpha \in [0, 1]}\{\alpha^r M_{r, r, \alpha^r}\} = \max_{\alpha \in [0, 1]}\{\alpha^r(1-\alpha)^r + r\alpha^{r+1}(1-\alpha)^{r-1}\}$.
\end{proof}

Accordingly, from \Cref{thm:2-avoiding-uniform} we get that $p^*(n, r, 2) \le p(n, r, 2) \leq (\gamma_r + o(1))\binom{n}{r}^2$ as $n \rightarrow \infty$, and hence the same is true for $d(n, r, 2)$ by \Cref{thm:disjoint-to-product}. This bound is asymptotically tight for these problems by considering the pair $\left( \mF_{X, 1}, \mF_{X^c, 0}\right)$ where $|X| = \alpha n$, and $\alpha \in [0, 1]$ yields the maximum value of $\gamma_r$, as above.

\section{Final Remarks and Open Problems}\label{sec:final}

We have addressed a variety of problems concerning the maximum number of disjoint pairs in set systems with certain intersection 
conditions. Many problems remain open. For example, \Cref{thm:t-avoiding-single} shows that the family $\mF^*(n, t)$ (see \Cref{sec:intro}) that maximizes the size of $t$-avoiding set systems for $n$ sufficiently large also is asymptotically optimal for maximizing the number of disjoint pairs. 
We conjecture that $\mF^*(n, t)$ indeed maximizes the number of disjoint pairs among all $t$-avoiding set systems, for $n$ sufficiently large.
\begin{conj}\label{conj:conjecture1}
For every integer $t \ge 1$ there exists an integer $n_0 = n_0(t)$ such that the following holds.
If $n \ge n_0$ and $\mF \subset \mP[n]$ is $t$-avoiding, then 
\[
	d(\mF) \leq d\left(\mF^*(n, t)\right).
\]

\end{conj}

We further introduced the three functions
$d(n, r, t)$, $p(n, r, t)$, and $p^*(n, r, t)$, each of which turned out to be asymptotically equivalent to each other (see \Cref{sec:dp-uniform}). We made progress in determining $p(n, r, 1)$ (and hence also $d(n, r, 1)$ and $p^*(n, r, 1)$) for $n$ large, and also $p(n, r, 2)$, asymptotically.
The extremal constructions for all three of these problems turned out to be of the form $\left(\mF_{X, a}, \mF_{X^c, b}\right)$, for suitable 
$X \subset [n]$ and nonnegative integers $a, b$, as equal as possible. We conjecture that this phenomenon persists for all $t < r$.

\begin{conj}\label{conj:conjecture2}
Let $r$ and $t$ be positive integers with $t < r$. Then there exist nonnegative integers $a,b$ and $X\subset [n]$ such that 
\[
	p(n,r,t)=(1+o(1))|\mF_{X, a}||\mF_{X^{c}, b}|.
\]
\end{conj}
By \Cref{thm:disjoint-to-product}, \Cref{conj:conjecture2} would imply that $d(n, r, t) = (1+ o(1))d\left(\mF_{X, a}, \mF_{X^{c}, b}\right)$. Note that by \Cref{lem:technical1} we may pass from a $t$-cross-avoiding pair to a $\{0, \ldots , t-1\}$-cross-intersecting pair of set systems when attempting to prove 
\Cref{conj:conjecture2}.
When $t=2k+1$ is odd, a simple calculation shows that the product $|\mF_{X, a}||\mF_{X^{c}, b}|$ is maximized when $|X|=\lfloor n/2\rfloor$ and $a=b=k$. Hence, we expect the extremal construction to exhibit some symmetry when $t$ is odd. On the other hand, when $t$ is even we expect the extremal construction to be asymmetric, as evidenced by the optimal configuration in \Cref{thm:2-avoiding-uniform}. 
Note that in order to deal with this asymmetry, we relied on a result of Huang, Linial, Naves, Peled 
and Sudakov~\cite{HuangLinialNavesPeledSudakov}, concerning densities of red and blue cliques in $2$-edge-colourings of the complete graph. 
Rather vaguely, one way of tackling \Cref{conj:conjecture2} might be to give a suitable hypergraph generalization of their result.

Let us close the paper by mentioning a connection to isoperimetric problems. We believe that the pairs $\left(\mF_{X, a}, \mF_{X^c, b}\right)$ with $a + b \le t -1$ as equal as possible should be 
optimal for maximizing $p^*(n, r, t)$. For simplicity, let us specialize to the case when $r = 3$ and $t = 2$ (this case has a pleasant interpretation as the maximum product of monochromatic triangles in a $2$-edge-colouring of $K_n$). Thus, if $(\mA, \mB)$ is a pair of $3$-uniform $\{0, 1\}$-intersecting hypergraphs and $n$ is sufficiently large, we believe that the exact bound $|\mA||\mB| \le \gamma \binom{n}{3}^2$ should hold, where
$\gamma = \gamma_3 = \max_{\alpha \in [0, 1]}\{\alpha^3(1-\alpha)^3 + 3\alpha^4(1-\alpha)^2\}$. One way of establishing this might be to prove a lower bound on the \emph{lower-upper shadow}. Recall that the \emph{lower shadow} of a set system $\mF \subset [n]^{(r)}$, denoted $\p \mF$, is the set $\{A \in [n]^{(r-1)}: A \subset F, \text{ for some } F \in \mF\}$. The upper shadow is defined similarly, and denoted $\p^+\mF$.

\begin{qn}\label{qn:question}
Suppose that $\mA \subset [n]^{(3)}$ with $|\mA| = \binom{x}{3}$ for some real number $x \ge 3$. Is it true that 
\[
	|\p^+\p \mA| \ge \binom{x}{3} + \binom{x}{2}(n-x)?
\]

\end{qn} 

Let $(\mA, \mB)$ be a pair of $\{0, 1\}$-intersecting $3$-uniform set systems and write $|\mA| = \binom{x}{3}$ for some real $x \ge 3$. If \Cref{qn:question} is true, then, since $\mB \subset (\p^+\p \mA)^c$, we have that $|\mB| \le \binom{n-x}{3} + \binom{n-x}{2}x$, and hence $|\mA||\mB| \le \binom{x}{3}\left(\binom{n-x}{3} + \binom{n-x}{2}x\right)$. Setting $x = \alpha n$ and optimizing yields $|\mA||\mB| \le \gamma \binom{n}{3}^2$. Note that \Cref{qn:question} is related to several stronger conjectures made by Bollob{\'a}s and Leader~\cite{BollobasLeaderIso} concerning `mixed' shadows.

\section{Acknowledgements}
Part of this research was carried out at IMT Institute for Advanced Studies, Lucca. We are grateful for their generous hospitality. We would like to thank Andrew Thomason for bringing to our attention the paper of Huang et al., and Bela Bollob{\'a}s, for helpful comments. We thank the referees for many helpful suggestions that improved the presentation of this paper.

    \bibliography{disjointpairs_int}
    \bibliographystyle{amsplain}

\end{document}